\documentclass{article}
\usepackage{graphicx} 
\usepackage[a4paper,margin=1in]{geometry}

\title{On the Submodule Structure of Hook Specht Modules in Characteristic $2$ II}
\author{Zain Ahmed Kapadia \\ Queen Mary, University of London, Mile End Road, London E1 4NS, UK}

\usepackage{amsmath, mathtools, physics, array, tabularx, multirow, genyoungtabtikz, booktabs}

\usepackage{amsthm, amssymb,ytableau} 

\usepackage{hyperref, cleveref}

\newtheorem{theorem}{Theorem}[section]
\newtheorem*{theorem*}{Theorem}

\newtheorem*{conjecture*}{Conjecture}
\newtheorem{lemma}[theorem]{Lemma}
\newtheorem*{lemma*}{Lemma}

\newtheorem*{proposition*}{Proposition}
\newtheorem{corollary}[theorem]{Corollary}
\newtheorem*{corollary*}{Corollary}

\theoremstyle{definition}
\newtheorem{definition}[theorem]{Definition}
\newtheorem*{definition*}{Definition}
\newtheorem{example}[theorem]{Example}
\newtheorem*{example*}{Example}

\theoremstyle{remark}
\newtheorem{remark}[theorem]{Remark}
\newtheorem*{remark*}{Remark}

\newtheorem*{statement*}{Statement}

\usepackage[many]{tcolorbox}
\tcolorboxenvironment{theorem}{
  enhanced,
  borderline={0.8pt}{0pt}{black},
  borderline={0.4pt}{2pt}{black},
  boxrule=0.4pt,
  colback=white,
  coltitle=black,
  sharp corners,
}
\tcolorboxenvironment{theorem*}{
  enhanced,
  borderline={0.8pt}{0pt}{black},
  borderline={0.4pt}{2pt}{black},
  boxrule=0.4pt,
  colback=white,
  coltitle=black,
  sharp corners,
}
\tcolorboxenvironment{conjecture}{
  enhanced,
  borderline={0.8pt}{0pt}{black},
  boxrule=0.4pt,
  colback=white,
  coltitle=black,
  sharp corners,
}
\tcolorboxenvironment{conjecture*}{
  enhanced,
  borderline={0.8pt}{0pt}{black},
  boxrule=0.4pt,
  colback=white,
  coltitle=black,
  sharp corners,
}

\tcolorboxenvironment{lemma}{
    colback=gray!10,
    boxrule=0.1pt,
    before skip=1em,after skip=1em
}
\tcolorboxenvironment{lemma*}{
    colback=gray!10,
    boxrule=0.1pt,
    before skip=1em,after skip=1em
}

\tcolorboxenvironment{proposition}{
    colback=gray!10,
    boxrule=0.1pt,
    before skip=1em,after skip=1em
}
\tcolorboxenvironment{proposition*}{
    colback=gray!10,
    boxrule=0.1pt,
    before skip=1em,after skip=1em
}
\tcolorboxenvironment{corollary}{
    colback=gray!10,
    boxrule=0.1pt,
    before skip=1em,after skip=1em
}
\tcolorboxenvironment{corollary*}{
    colback=gray!10,
    boxrule=0.1pt,
    before skip=1em,after skip=1em
}

\tcolorboxenvironment{definition}{
	breakable,
    colback=gray!10,
    boxrule=0.1pt,
    before skip=1em,after skip=1em
}
\tcolorboxenvironment{definition*}{
	breakable,
    colback=gray!10,
    boxrule=0.1pt,
    before skip=1em,after skip=1em
}

\tcolorboxenvironment{example}{
	breakable,
    colback=white,
    borderline={0.8pt}{0pt}{black},
    boxrule=0.1pt,
    before skip=1em,after skip=1em
}
\tcolorboxenvironment{example*}{
	breakable,
    colback=gray!10,
    boxrule=0.1pt,
    before skip=1em,after skip=1em
}

\tcolorboxenvironment{remark}{
	breakable,
	blanker,
    before skip=1em,after skip=1em
}
\tcolorboxenvironment{remark*}{
	breakable,
	blanker,
    before skip=1em,after skip=1em
}
\tcolorboxenvironment{statement}{
	breakable,
	blanker,
    before skip=1em,after skip=1em
}
\tcolorboxenvironment{statement*}{
	breakable,
	blanker,
    before skip=1em,after skip=1em
}

\usepackage{tikz, tikz-cd}
\usetikzlibrary{cd, graphs, positioning, fit} 

\usepackage{float} 

\usepackage{biblatex} 
\newcommand{\isom}{\cong} 

\begin{document}

\maketitle

\begin{abstract}
    We classify which $2$-part Young modules in characteristic $2$ are uniserial, and which hook Specht modules in characteristic $2$ are direct sums of uniserial summands. This is a continuation of the author's previous work \cite{kapadia2024submodule}.
\end{abstract}

\tableofcontents

\newpage
\section{Introduction}\label{Introduction}

Let $n$ be a positive integer and let $S_n$ be the symmetric group on $n$ letters. There are many important families of modules in the representation theory of the symmetric group. For example, the \emph{Young permutation modules} $M_{\mathbb{F}}^\lambda,$ which are permutation modules on the cosets of the Young subgroups $S_\lambda;$ \emph{Young modules} $Y_{\mathbb{F}}^\mu$ which are the indecomposable summands of $M_{\mathbb{F}}^\lambda;$ and \emph{Specht modules} $S_{\mathbb{F}}^\nu$ which, in characteristic $0,$ are the irreducible representations of $S_n.$ In characteristic $0,$ Specht modules coincide with Young modules, but in positive characteristic this is no longer the case, with neither being (in general) simple and some Specht modules even being decomposable. Despite this, there remains in positive characteristic an intimate relationship between the structures of Young modules and Specht modules.

For example, \cite{donkin1987schur} showed that Young modules have a filtration whose quotients are isomorphic to Specht modules, and \cite{DonkinGeranios20201} showed that if $\lambda$ is of the form $(a, m-1, m-2, \ldots, 2, 1^b)$ with $a \geq m, b \geq 1, a-m$ even and $b$ odd, then $S_2^\lambda$ decomposes into a direct sum of Young modules labelled by partitions $\mu$ of the form $(x + m, y + m -1, m-2, \ldots, 1)$ where $x$ and $y$ satisfy some conditions. 

This paper continues the author's previous work \cite{kapadia2024submodule} in studying the submodule structure of hook Specht modules in characteristic $2;$ that is, $S_2^\lambda$ where $\lambda$ is of the form $(n-r,1^r)$ for $0 \leq r \leq n-1.$ The key result was the classification of uniserial hook Specht modules in characteristic $2.$ It is known that if $n$ is odd and $n-2r - 1 \not\equiv 0 \mod 2^L,$ where $2^{L-1} \leq r < 2^L,$ then $S_2^\lambda$ is decomposable \cite[Theorems 4.1, 4.5]{murphy1980decomposability}. This paper takes results primarily from \cite{donkin1987schur}, \cite{DonkinGeranios20201} and \cite{fayers_2002_schur_subalgebras_II} to classify which hook Specht modules in characteristic $2$ are direct sums of uniserial summands. Along the way, we also prove a nice result which states which $2$-part Young modules $Y_2^\mu$ are uniserial.

We now briefly indicate the layout of this paper. \Cref{Background} contains background definitions and results regarding partition combinatorics, Specht modules, Young modules and James modules. \Cref{section Uniserial 2-part Young modules} gives a sufficient and necessary condition for a $2$-part Young module to be uniserial. \Cref{Hook Specht modules which are direct sums of uniserial summands} concludes the paper by classifying which hook Specht modules are direct sums of uniserial summands.

The author acknowledges his PhD supervisor, \href{https://webspace.maths.qmul.ac.uk/m.fayers/}{Dr Matthew Fayers}, for the ideas and guidance provided throughout every step of this paper; Stephen Donkin and Dr Haralampos Geranios for friendly and insightful discussions; and the \href{https://www.ukri.org/councils/epsrc/}{Engineering and Physical Sciences Research Council (EPSRC)} for funding this ongoing PhD project.

\newpage
\section{Background}\label{Background}

In this section we state introductory definitions and results, directing the reader to appropriate sources for more details where necessary.

Fix a non-negative integer $n.$ A \emph{partition} $\lambda$ of $n$ is a weakly decreasing sequence of non-negative integers $\lambda = (\lambda_1, \lambda_2, \ldots)$ such that $\sum_{i\geq 1}\lambda_i = n,$ which we denote by $\lambda \vdash n.$ We will typically write $\lambda$ without trailing zeroes, and in multiplicative notation. For example, the partition $(7,3,3,0, \ldots)$ will be denoted $(7,3^2).$ There is a partial order on the set of partitions of $n$ as follows: we say $\lambda \trianglerighteq \mu$ ($\lambda$ `\emph{dominates}' $\mu$) if and only if for all $j, \sum_{i = 1}^j \lambda_i \geq \sum_{i=1}^j \mu_i.$

In characteristic $0,$ the irreducible representations of $S_n,$ the \emph{Specht modules} $S_{\mathbb{F}}^\lambda,$ are labelled by partitions $\lambda$ of n, and \{$S_{\mathbb{F}}^\lambda \, | \, \lambda \vdash n$\} is a complete set of non-repeating irreducible representations for $S_n.$ These can be reduced modulo $p,$ for $p$ prime, to get Specht modules over fields of positive characteristic. These are, in general, no longer simple. We direct the reader to \cite{James1978} for more details on Specht modules.

For a partition $\lambda$ of a fixed $n,$ the \emph{Young subgroup} $S_\lambda \isom S_{\lambda_1} \times S_{\lambda_2} \times \cdots \times S_{\lambda_r} \leq S_n$ acts on disjoint subsets of $\{1, 2, \ldots, n\}.$ Over a field $\mathbb{F},$ there is a transitive permutation module $M_{\mathbb{F}}^\lambda$ on the cosets of $S_\lambda$ in $S_n.$ By using James' Submodule Theorem \cite[(4.8) p. 15]{James1978}, we get that \cite[p. 204]{erdmann2001young} there is a unique indecomposable summand of $M_{\mathbb{F}}^\lambda$ which contains $S_{\mathbb{F}}^\lambda,$ which we denote by $Y_{\mathbb{F}}^\lambda,$ and call it the \emph{Young module} associated to $\lambda.$ In particular, every indecomposable direct summand of $M_{\mathbb{F}}^\lambda$ is isomorphic to $Y_{\mathbb{F}}^\mu$ for some $\mu$ such that $\mu \trianglerighteq \lambda;$ precisely one summand of $M_{\mathbb{F}}^\lambda$ is isomorphic to $Y_{\mathbb{F}}^\lambda;$ and if two Young modules are isomorphic then they are labelled by the same partition. Furthermore, $Y_{\mathbb{F}}^\lambda$ is self-dual \cite[p.4]{elkin2020young}. We direct the reader to \cite{erdmann2001young} for more details on Young modules.

We say that a partition $\lambda$ is \emph{$p$-regular} if there is \emph{no} $i \in \mathbb{Z}$ such that $\lambda_i = \lambda_{i+1} = \cdots = \lambda_{i+p-1} > 0$, and we say that $\lambda$ is \emph{$p$-singular} if there is such an $i$. If the characteristic of $\mathbb{F}$ is $p > 0$ and $\lambda$ is $p$-regular, one can show that $S_{\mathbb{F}}^\lambda$ has a unique maximal submodule. We label the quotient by this maximal submodule as $D_{\mathbb{F}}^\lambda,$ sometimes called the \emph{James module} associated to $\lambda.$ One can show that $\{D_{\mathbb{F}}^\lambda \, | \, \lambda \vdash n, \lambda \text{ is $p$-regular}\}$ is a complete set of non-repeating irreducible modular representations for $S_n.$ Again, we direct the reader to \cite{James1978} for more details on James modules.

We may write $M^\lambda, Y^\lambda, S^\lambda$ or $D^\lambda$ if the ground field is clear from context, and/or if the result is true for arbitrary ground fields. If the choice of ground field is arbitrary up to the characteristic $p$ of $\mathbb{F},$ we may write $M_p^\lambda, Y_p^\lambda, S_p^\lambda$ and $D_p^\lambda.$

If, for $\lambda \vdash n,$ we have $\lambda_3 = 0,$ we say that $\lambda$ is a \emph{$2$-part partition}, and we call $S^\lambda$ (resp. $Y^\lambda$) a \emph{$2$-part Specht module} (resp. \emph{$2$-part Young module}). Similarly, if $\lambda \vdash n$ is of the form $(n-r, 1^r)$ for $0 \leq r \leq n-1,$ we say that $\lambda$ is a \emph{hook partition}, and we call $S^\lambda$ a \emph{hook Specht module}. Note that the partition $\lambda = (n)$ is considered both a $2$-part partition and a hook partition, and so $S^{(n)}$ is also a $2$-part Specht module and a hook Specht module. This paper will focus primarily on $2$-part Specht modules, $2$-part Young modules, hook Specht modules, and the relationship between them all.

\newpage
\section{Uniserial $2$-part Young modules}\label{section Uniserial 2-part Young modules}

It is known that odd hook Specht modules in characteristic $2$ decompose as a direct sum of Young modules labelled by $2$-part partitions \cite[Proposition 7.1.1]{DonkinGeranios20201}. To classify which hook Specht modules in characteristic $2$ are direct sums of uniserial summands, we first need to study which $2$-part Young modules in characteristic $2$ are uniserial. In this section we give a sufficient and necessary condition for this to be the case. We first state some definitions, and then refine a theorem classifying uniserial $2$-part Specht modules in characteristic $2$ from the author's previous paper. 

\begin{definition}\label{contains}
    Let $a$ and $b$ be non-negative integers. Let $a = \sum_{i\geq 0}a_ip^i$ and $b=\sum_{j\geq 0}b_jp^j$ be their $p$-adic expansions. We write $a \supseteq_p b$ \emph{``$a$ contains $b$''} if and only if $b_i \in \{0, a_i\}$ for every $i$. 
\end{definition}

\begin{definition}\cite[p. 525]{murphy1982submodule}
    For a non-negative integer $r$ and a prime $p,$ define $L_p(r)$ to be the smallest non-negative integer such that $r < p^{L_p(r)}.$
\end{definition}
One can think of $L_p(r)$ as the number of significant digits when $r$ is written in its $p$-adic expansion.

\begin{theorem}\cite[Theorem 3.13]{kapadia2024submodule}\label{2_part_uniserial_old}
    Let $\lambda = (\lambda_1, \lambda_2) \vdash n.$ If $\alpha := \lambda_1 - \lambda_2 + 1$ has at least two non-zero digits in its binary expansion, define $a := \nu_2(\alpha), b := \nu_2(\alpha + 2^{\nu_2(\alpha)}),$ and $ c := \nu_2(\alpha - 2^{\nu_2(\alpha)}).$ That is, $a$ is the first $1$ in the binary expansion of $\alpha,$ $b$ is the first $0$ after $a,$ and $c$ is the first $1$ after $a.$
   
    Then $S^\lambda_2$ is uniserial if and only if 
    \begin{enumerate}
        \item $\alpha$ is a power of two or;
        \item $\alpha$ is not a power of two and $c > b$ and $2^c > \lambda_2$ or;
        \item $\alpha$ is not a power of two and $c < b$ and $2^b + 2^c > \lambda_2.$
    \end{enumerate}
\end{theorem}

We can reformulate the above theorem into something easier to work with:

\begin{corollary}\label{2_part_uniserial_new}
    Let $\lambda = (\lambda_1, \lambda_2) \vdash n, \alpha := \lambda_1 - \lambda_2 + 1, L := L(\lambda_2),$ and let $\nu := \nu_2(\alpha).$ Then the Specht module $S_2^{\lambda}$ is uniserial if and only if one of the following occurs:
    \begin{enumerate}
        \item $\alpha + 2^\nu \equiv 0 \mod 2^L;$ or
        \item $\alpha + 2^\nu \equiv 2^{L-1} \mod 2^L,$ and $2^{\nu+1} + 2^{L-1} > \lambda_2;$ or
        \item $\alpha - 2^\nu \equiv 0 \mod 2^L.$
    \end{enumerate}
\end{corollary}

\begin{proof}
    First assume $S_2^\lambda$ is uniserial. Then by \Cref{2_part_uniserial_old}, we must be in one of three cases. Take the definitions of $a, b$ and $c$ as in the statement of the theorem.

    If $\alpha$ is a power of two, then $\alpha = 2^\nu,$ and so $\alpha - 2^\nu = 0,$ hence $\alpha - 2^\nu \equiv 0 \mod 2^L.$

    In the case that $\alpha$ is not a power of two, $c>b,$ and $2^c > \lambda_2,$ we have that $\nu_2(\alpha - 2^\nu) = \nu_2(\alpha - 2^a) = c.$ So $2^c | (\alpha - 2^\nu),$ and, as $2^c > \lambda_2,$ we have $2^c \geq 2^L,$ so $\alpha - 2^\nu \equiv 0 \mod 2^L.$ 

    In the case that $\alpha$ is not a power of two, $c<b$ and $2^b + 2^c > \lambda_2,$ we have that $\nu_2(\alpha + 2^\nu) = \nu_2(\alpha + 2^a) = b.$ If $b \geq L$ then $\alpha + 2^\nu \equiv 0 \mod 2^L.$ If, $b < L,$ we must have that $b = L-1,$ otherwise we would have $b \leq L-2,$ hence $2^b + 2^c < 2^b + 2^b = 2^{b+1} \leq 2^{L-1} \leq \lambda_2,$ a contradiction to the assumption that $2^b+2^c > \lambda_2.$ Hence we must have $b=L-1,$ and so $\alpha + 2^\nu \equiv 2^{L-1} \mod 2^L$ and $2^{\nu+1} + 2^{L-1} > \lambda_2,$ as $c = \nu+1.$

    Now for the converse direction. If $\alpha$ is a power of two, then $S_2^\lambda$ is uniserial. So assume now that $\alpha$ is not a power of two, and define $a := \nu_2(\alpha), b := \nu_2(\alpha + 2^a)$ and $c := \nu_2(\alpha - 2^a)$ as in the statement of \Cref{2_part_uniserial_old}.

    Assume $\alpha + 2^\nu \equiv 0 \mod 2^L.$  By definition, $2^b | (\alpha + 2^\nu),$ so $b \geq L.$ We also have that either $\nu+1 = b$ or $\nu+1 = c.$ In the case $\nu+1=b,$ we must have $c > b$ and so $2^c > 2^b \geq 2^L > \lambda_2,$ hence $S_2^\lambda$ is uniserial. In the case $\nu+1 = c,$ we have $c < b$ and $2^b + 2^c > 2^b \geq 2^L > \lambda_2,$ so $S_2^\lambda$ is uniserial.

    Now consider the case that $\alpha + 2^\nu \equiv 2^{L-1} \mod 2^L$ and $2^{\nu+1} + 2^{L-1} > \lambda_2.$ Again, by the definition of $b,$ we have $2^b \equiv 2^{L-1} \mod 2^L$ and hence $b = L-1.$ Again, we have that either $\nu+1 = b$ or $\nu+1 = c.$ In the case $\nu+1=b,$ we must have $c > b$ and so $2^c \geq 2^{b+1} = 2^{L} > \lambda_2,$ hence $S_2^\lambda$ is uniserial. In the case $\nu+1 = c,$ we have $c < b$ and $2^b + 2^c = 2^{L-1} + 2^{\nu+1} > \lambda_2$ by assumption, so $S_2^\lambda$ is uniserial.

    Now for the case that $\alpha - 2^\nu \equiv 0 \mod 2^L.$ By definition, $2^c |(\alpha - 2^\nu),$ so $c \geq L.$ We also have that either $\nu+1 = b$ or $\nu+1 = c.$ In the case that $\nu+1=b,$ we must have $c>b$ and so $2^c \geq 2^L > \lambda_2$ hence $S_2^\lambda$ is uniserial. In the case $\nu+1 = c,$ we have $c < b$ and $2^c + 2^b  > 2^c \geq 2^L > \lambda_2,$ so $S_2^\lambda$ is uniserial, as required.
\end{proof}

For partitions $\lambda,\mu$ of $n,$ we take $\nabla(\mu)$ to be the dual Weyl module and $L(\lambda)$ to be the simple module for the general linear groups, as introduced and used in \cite{martin1993schur} (where $\nabla(\mu)$ is called $M(\mu)$) and \cite{green2006polynomial} (where $\nabla(\mu)$ is called $D_{\mu}$). The precise definitions and constructions of these modules are not required for this paper.

\begin{definition}
    For $\lambda, \mu \vdash n,$ we write $[\nabla(\mu): L(\lambda)]$ for the composition multiplicity of $L(\lambda)$ inside of $\nabla(\mu).$
\end{definition}

We have the following theorem which states that Young modules have a filtration via Specht modules.

\begin{theorem}\label{Young module filtered by Specht SD}\cite[(2.6) p. 360]{donkin1987schur}
    Let $\lambda \vdash n$ be a partition and $Y_\mathbb{F}^\lambda$ be the Young module for $\lambda$ over a field $\mathbb{F}.$ Then $Y_\mathbb{F}^\lambda$ has a filtration $0 = Y_0 \leq Y_1 \leq \ldots \leq Y_s = Y_\mathbb{F}^\lambda$ for some $s,$ where $Y_i / Y_{i-1}$ is isomorphic to a direct sum of $[\nabla(\mu^i): L(\lambda)]$ many copies of $S_\mathbb{F}^{\mu^i},$ with the labelling chosen so that $i < j$ if $\mu^i \trianglelefteq \mu^j.$
\end{theorem}

Fix some $i$ and write $\mu$ for $\mu^i$. As $Y_{\mathbb{F}}^\lambda$ is the unique direct summand of $M_{\mathbb{F}}^\lambda$ which contains $S_{\mathbb{F}}^\lambda,$ we can take $Y_1 = Y_1 / Y_0 = S_{\mathbb{F}}^\lambda.$ By the above theorem, we also have that $\lambda \trianglelefteq \mu$ for all $\mu.$ In the case that $\lambda$ is a $2$-part partition, this implies that $\mu$ is also a $2$-part partition. A recursive formula to calculate $[\nabla(\mu): L(\lambda)]$ in this case is given below.

\begin{theorem}\label{decomp number recursive formula}\cite[p. 321]{fayers_2002_schur_subalgebras_II}
    Let $\mu = (\mu_1, \mu_2)$ and $\lambda = (\lambda_1,\lambda_2)$ be $2$-part partitions of a positive integer $n.$ Define $\alpha := \mu_1 - \mu_2 + 1$ and $\beta := \lambda_1 - \lambda_2 + 1$ (note that this is different to the literature). Write $d(\alpha,\beta) := [\nabla(\mu) : L(\lambda)].$ Then we have the recursive formula:

    \[ d(\alpha,\beta) = \begin{cases} 
      d(\frac{\alpha+1}{2}, \frac{\beta+1}{2}) & \text{ if $\alpha$ and $\beta$ are odd and congruent mod $4$ } \\
      d(\frac{\alpha-1}{2}, \frac{\beta+1}{2}) & \text{ if $\alpha$ and $\beta$ are odd but not congruent mod $4$ } \\
      d(\frac{\alpha}{2}, \frac{\beta}{2}) & \text{ if $\alpha$ and $\beta$ are even and congruent mod $4$ } \\
      0 & \text{ if $\alpha$ and $\beta$ are even but not congruent mod $4$ } 
   \end{cases}
    \]
\end{theorem}

We give a closed form of the recursive formula in the following corollary:

\begin{corollary}\label{closed form for d}
    $d(\alpha,\beta) = 1$ if and only if $\alpha + \beta - 2 \supseteq_2 \alpha - \beta.$
\end{corollary}
\begin{proof}
    We prove this via induction on $\alpha$ and $\beta.$ In the base case, $d(1,1) = 1$ \cite[Corollary 4.17]{mathas1999iwahori}, and $1+1 - 2 \supseteq_2 1-1.$

    Now assume that up to some fixed $\alpha$ and $\beta,$ we have $d(\alpha,\beta) = 1$ if and only if $\alpha + \beta - 2 \supseteq_2 \alpha - \beta.$ 

    Now consider $d(\alpha+1,\beta).$ We have four cases:
    \begin{enumerate}
        \item If $\alpha+1$ and $\beta$ are both odd and congruent mod $4,$ then $d(\alpha + 1,\beta) = d(\frac{\alpha+2}{2}, \frac{\beta+1}{2}),$ which equals $1$ if and only if $\frac{\alpha + \beta -1}{2} \supseteq_2 \frac{\alpha - \beta + 1}{2}$ (by the induction hypothesis), which equals $1$ if and only if $(\alpha + 1) + \beta - 2 \supseteq_2 (\alpha+1) - \beta$ as required.
        
        \item If $\alpha+1$ and $\beta$ are both odd but not congruent mod $4,$ then $d(\alpha+1,\beta) = d(\frac{\alpha}{2}, \frac{\beta+1}{2}),$ which equals $1$ if and only if $\frac{\alpha + \beta -3}{2} \supseteq_2 \frac{\alpha - \beta - 1}{2}$ (by the induction hypothesis), which equals $1$ if and only if $(\alpha+1) + \beta -4 \supseteq_2 (\alpha+1) - \beta - 2.$ As $\alpha+1$ and $\beta$ are both odd but not congruent mod $4,$ both sides of the relation are $0$ mod $4,$ hence one can add $2$ to both sides and the statement is still true. Hence $d(\alpha+1,\beta) = 1$ if and only if $(\alpha+1) + \beta - 2 \supseteq_2 (\alpha+1) - \beta$ as required.
        
        \item If $\alpha+1$ and $\beta$ are both even and congruent mod $4,$ then $d(\alpha+1,\beta) = d(\frac{\alpha+1}{2},\frac{\beta}{2}),$ which is $1$ if and only if $\frac{\alpha + \beta - 3}{2} \supseteq_2 \frac{\alpha + 1 - \beta}{2}$ (by the induction hypothesis), which is $1$ if and only if $(\alpha + 1) + \beta - 4 \supseteq_2 (\alpha+1) - \beta.$ As $\alpha+1$ and $\beta$ are both even and congruent mod $4,$ both sides of the relation are $0$ mod $4,$ hence one can add $2$ to the left hand side of the relation and the statement is still true. Hence $d(\alpha+1, \beta) = 1$ if and only if $(\alpha+1) + \beta - 2 \supseteq_2 (\alpha+1) - \beta$ as required. 
        
        \item If $\alpha+1$ and $\beta$ are both even and not congruent mod $4,$ then $d(\alpha+1,\beta) = 0.$ Similarly, $(\alpha + 1) + \beta + 2 \not\supseteq_2 (\alpha+1) - \beta$ as the left hand side is $0$ mod $4,$ but the right hand side is $2$ mod $4.$  
    \end{enumerate}

    Now consider $d(\alpha,\beta+1).$ As before, we have four cases:
    \begin{enumerate}
        \item If $\alpha$ and $\beta+1$ are both odd and congruent mod $4,$ then $d(\alpha,\beta+1) = d(\frac{\alpha+1}{2},\frac{\beta+2}{2}),$ which equals $1$ if and only if $\frac{\alpha + \beta -1}{2} \supseteq_2 \frac{\alpha - \beta -1}{2}$ (by the induction hypothesis), which equals $1$ if and only if $\alpha + (\beta+1) - 2 \supseteq_2 \alpha - (\beta+1)$ as required.

        \item If $\alpha$ and $\beta+1$ are both odd but not congruent mod $4,$ then $d(\alpha,\beta+1) = d(\frac{\alpha-1}{2},\frac{\beta+2}{2})$ which is $1$ if and only if $\frac{\alpha + \beta -3}{2} \supseteq_2 \frac{\alpha - \beta  -3}{2}$ (by the induction hypothesis), which equals $1$ if and only if $\alpha + (\beta+1) - 4 \supseteq_2 \alpha - (\beta+1) - 2.$ As $\alpha$ and $\beta+1$ are both odd but not congruent mod $4,$ both sides of the relation are $0$ mod $4,$ hence one can add $2$ to both sides and the statement is still true. Hence $d(\alpha,\beta+1) = 1$ if and only if $\alpha + (\beta+1) - 2 \supseteq_2 \alpha - (\beta+1)$ as required.

        \item If $\alpha$ and $\beta+1$ are both even and congruent mod $4,$ then $d(\alpha,\beta+1) = d(\frac{\alpha}{2},\frac{\beta+1}{2})$ which is $1$ if and only if $\frac{\alpha + \beta - 3}{2} \supseteq_2 \frac{\alpha - \beta - 1}{2}$ (by the induction hypothesis), which equals $1$ if and only if $\alpha + (\beta+1) - 4 \supseteq_2 \alpha - (\beta+1).$ As $\alpha$ and $\beta+1$ are both even and congruent mod $4,$ both sides of the relation are $0$ mod $4,$ hence one can add $2$ to the left hand side of the relation and the statement is still true. Hence $d(\alpha,\beta+1) = 1$ if and only if $\alpha + (\beta+1) - 2 \supseteq_2 \alpha - (\beta+1).$

        \item If $\alpha$ and $\beta+1$ are both even and not congruent mod $4,$ then $d(\alpha,\beta+1) = 0.$ Similarly, $\alpha + (\beta+1) - 2 \not\supseteq_2 \alpha - (\beta+1)$ as the left hand side if $0$ mod $4,$ but the right hand side is $2$ mod $4.$ 
    \end{enumerate}

    Hence the statement is true by induction.
\end{proof}

We will use the following notation frequently:

\begin{definition}
    For $\lambda = (\lambda_1, \lambda_2) \vdash n$ a $2$-part partition, and non-negative integer $d \leq \lambda_2,$ we define $\lambda \pm d := (\lambda_1 + d, \lambda_2 - d).$
\end{definition}

We can now rewrite the statement of the filtration of $2$-part Young modules via $2$-part Specht modules in a nice form.

\begin{corollary}\label{Young module filtered by Specht Zain}
    Let $n$ be a positive integer, let $\lambda = (\lambda_1, \lambda_2)$ be a $2$-part partition of $n,$ and let $\alpha := \lambda_1 - \lambda_2 + 1.$ Then $Y_2^\lambda$ has a filtration via Specht modules, with the set of factors being \[ \{ S_2^{(\lambda \pm d)} \, | \, 0 \leq d \leq \lambda_2, \alpha -1 + d \supseteq_2 d \},\] and the factors going from bottom to top in increasing value of $d.$
\end{corollary}
\begin{proof}
     By \Cref{Young module filtered by Specht SD} and \Cref{closed form for d}, this follows immediately as $S_2^{(\lambda \pm d)}$ appears in the filtration of $Y_2^\lambda$ if and only if $d(\alpha + 2d,\alpha) = 1,$ which is true if and only if $2\alpha + 2d -2 \supseteq_2 2d,$ if and only if $\alpha -1 + d \supseteq_2 d. $
\end{proof}

\begin{remark}
    An equivalent statement to $\alpha - 1 + d \supseteq_2 d$ is the statement that $(\alpha - 1) \cap_2 d = 0,$ where $(\alpha - 1) \cap_2 d$ is the bitwise AND operator on $(\alpha - 1)$ and $d$ when written in binary.
\end{remark}

Given that $2$-part Young modules are filtered by $2$-part Specht modules, and the result which states precisely which $2$-part Specht modules are uniserial, we can work towards a proof classifying the uniserial $2$-part Young modules. 

We can do the $2$-singular case immediately.

\begin{theorem}\label{Young uniserial 2-singular}
    Let $\lambda = (\lambda_1, \lambda_1)$ be a $2$-singular $2$-part partition of $n.$ Then $Y_2^\lambda$ is uniserial if and only if $\lambda_1 = 1.$
\end{theorem}
\begin{proof}
    First observe that $\alpha = \lambda_1 - \lambda_1 + 1 = 1,$ so $\alpha - 1 + d \supseteq_2 d$ if and only if $d \supseteq_2 d,$ so by \Cref{Young module filtered by Specht Zain} we have that $Y_2^\lambda$ is filtered by $S_2^{(\lambda \pm d)}$ for all $0 \leq d \leq \lambda_1.$

    Now assume $\lambda_1 = 1,$ then by \Cref{Young module filtered by Specht Zain}, $Y_2^{(1,1)}$ is filtered from bottom to top by $S_2^{(1,1)}$ and $S_2^{(2)}.$ Both are isomorphic to $D_2^{(2)},$ and $Y_2^{(1,1)}$ is indecomposable by definition, so $Y_2^{(1,1)}$ is uniserial.

    Now assume $\lambda_1 >1$. We will show that $Y_2^\lambda$ is not uniserial. Assume that $\lambda_1 \not\equiv -1 \mod 2^L,$ and assume for contradiction that $Y_2^\lambda$ is uniserial. By \Cref{Young module filtered by Specht Zain}, $S_2^{(\lambda \pm \lambda_1)} = S_2^{(n)} = D_2^{(n)}$ appears in the head of $Y_2^\lambda,$ hence by self-duality must appear in the socle of $Y_2^\lambda.$ But, by assumption, $Y_2^\lambda$ is uniserial, so $D_2^{(n)}$ appears in the socle of the submodule isomorphic to $S_2^{\lambda}.$ However, by \cite[Theorem 24.4]{James1978}, this is true if and only if $\lambda_1 \equiv -1 \mod 2^{L(\lambda_2)},$ a contradiction. Hence $Y_2^\lambda$ is not uniserial.

    Now assume that $\lambda_1 > 1,$ and that $\lambda_1 \equiv -1 \equiv 1 + 2 + \ldots + 2^{L-1} \mod 2^L.$  Again, assume for contradiction that $Y_2^\lambda$ is uniserial. Again by \Cref{Young module filtered by Specht Zain}, we know that the second composition factor of $Y_2^\lambda$ when read from top to bottom must be $D_2^{(n-1,1)}.$ By self duality, this must be the second composition factor of $Y_2^\lambda$ when read from bottom to top. As $\lambda_1 > 1,$ this must be a composition factor of $S_2^\lambda.$ However, $\alpha + 2(\lambda_1 - 1) \not\supseteq \lambda_1 -1,$ as the right hand side of the relation is congruent to $2$ mod $4,$ but the left hand side of the relation is congruent to $1$ mod $4,$ so $D_2^{(n-1,1)}$ is not a composition factor of $S_2^\lambda$ \cite[Corollary 3.10]{kapadia2024submodule} a contradiction. Hence $Y_2^\lambda$ is not uniserial.
\end{proof}

For the $2$-regular case, some preliminary results are needed first.

\begin{lemma}\label{Young module uniserial if composition factors}
    Let $\lambda = (\lambda_1, \lambda_2)$ be a $2$-regular partition of $n$, and assume $S_2^\lambda$ is uniserial. Then $Y_2^\lambda$ is uniserial if and only if $Y_2^\lambda / S_2^\lambda$ has the same set of composition factors, including multiplicities, as $S_2^\lambda,$ excluding $D_2^\lambda.$
\end{lemma}
\begin{proof}
    First assume that $Y_2^\lambda$ is uniserial. Then by definition there is a unique composition series $\{0\} = Y_0 \subseteq Y_1 \subseteq \ldots \subseteq Y_k = Y_2^\lambda.$ Let $D_i := Y_i / Y_{i-1},$ and we know there is some submodule $Y_j$ which is isomorphic to $S_2^\lambda$ . By \Cref{Young module filtered by Specht Zain}, we know that $Y_2^\lambda$ is filtered by $S_2^\lambda$ and Specht modules labelled by partitions which dominate $\lambda,$ hence $D_2^\lambda$ appears as a composition factor of $Y_2^\lambda$ exactly once. Hence $D_j$ is isomorphic to $D_2^\lambda$ as $\lambda$ is $2$-regular, and by self-duality of $Y_2^\lambda$, $D_{j-l}$ is isomorphic to $D_{j+l}$ for all $1 \leq l \leq j-1,$ hence $k = 2j-1.$ Therefore $Y_2^\lambda / S_2^\lambda$ is isomorphic to the dual of the radical of $S_2^\lambda,$ and so has all of the composition factors of $S_2^\lambda$ excluding $D_2^\lambda,$ with the same multiplicities (precisely one each).

    Conversely, assume that $S_2^\lambda$ is uniserial and that $Y_2^\lambda / S_2^\lambda$ has the same composition factors as $S_2^\lambda$ excluding $D_2^\lambda,$ with each composition factor appearing exactly once. We know that $Y_2^\lambda$ is self-dual, and $[Y_2^\lambda : D_2^\lambda] = 1,$ so $Y_2^\lambda$ has an odd number of composition factors. Hence, we can write any composition series as $\{0\} = Y_0 \subseteq Y_1 \subseteq \ldots \subseteq Y_{2j-1} = Y_2^\lambda,$ where $j$ is the number of composition factors of $S_2^\lambda.$ Writing $D_i := Y_i / Y_{i-1},$ we must have that $D_j$ is isomorphic to $D_2^\lambda.$ There is a submodule inside $Y_2^\lambda$ isomorphic to $S_2^\lambda,$ and so it must contain $Y_j,$ but $S_2^\lambda$ and $Y_j$ have the same number of composition factors, so $Y_j$ is isomorphic to $S_2^\lambda.$ By self-duality of $Y_2^\lambda,$ this means $D_{j-l}$ is isomorphic to $D_{j+l}$ for all $1 \leq l \leq j-1,$ and hence $Y_2^\lambda$ has a unique composition series.

\end{proof}

The following corollary will be useful.

\begin{corollary}\label{Young module not a composition factors of the Specht}
    Let $\lambda = (\lambda_1, \lambda_2)$ be a $2$-regular partition of $n.$ Let $\alpha := \lambda_1 - \lambda_2 + 1.$ If there is a $0 \leq d \leq \lambda_2$ such that $\alpha -1 + d \supseteq_2 d$ but $\alpha + 2d \not\supseteq_2 d,$ then $Y_2^\lambda$ is not uniserial.
\end{corollary}
\begin{proof}
    As $\alpha -1 + d \supseteq_2 d,$ we have by \Cref{Young module filtered by Specht Zain}, $S_2^{(\lambda \pm d)}$ appears in the filtration of $Y_2^\lambda,$ hence $D_2^{(\lambda \pm d)}$ is a composition factor of $Y_2^\lambda.$ However, $\alpha + 2d \not\supseteq_2 d,$ so by \cite[Theorem 3.10]{kapadia2024submodule}, we have that $D_2^{(\lambda \pm d)}$ is not a composition factor of $S_2^\lambda,$ so we must have that $D_2^{\lambda}$ is a composition factor of $Y_2^\lambda / S_2^\lambda$ which is not a composition factor of $S_2^\lambda.$ Hence, by \Cref{Young module uniserial if composition factors}, $Y_2^\lambda$ is not uniserial.
\end{proof}

\begin{theorem}\label{Uniserial 2-part Young Modules}
    Let $\lambda = (\lambda_1, \lambda_2)$ be a $2$-part partition of $n, \alpha := \lambda_1 - \lambda_2 + 1, \nu := \nu_2(\alpha),$ and let $L := L_2(\lambda_2).$ Then the Young module in characteristic $2, Y_2^\lambda$ is uniserial if and only if $\alpha + 2^\nu \equiv 0$ mod $2^L$.
\end{theorem}
\begin{proof}

A necessary condition is that $S_2^\lambda,$ which appears as a submodule in $Y_2^\lambda,$ must be uniserial. By \Cref{2_part_uniserial_new}, $\alpha$ must be in one of three cases:
\begin{enumerate}
    \item $\alpha + 2^\nu \equiv 0 \mod 2^L;$
    \item $\alpha + 2^\nu \equiv 2^{L-1} \mod 2^L,$ and $2^{\nu+1} + 2^{L-1} > \lambda_2;$ or
    \item $\alpha - 2^\nu \equiv 0 \mod 2^L.$
\end{enumerate}

We work through each of the cases and show that only the first case leads to a uniserial Young module. 

We start with the case that $\lambda_1 > \lambda_2.$ By \Cref{Young module uniserial if composition factors}, it suffices to check when $Y_2^\lambda / S_2^\lambda$ has the same composition factors as $S_2^\lambda,$ excluding $D_2^\lambda.$

\begin{enumerate}
    \item Assume $\alpha + 2^\nu \equiv 0 \mod 2^L.$ We split this into the case that $\nu \geq L$ and $\nu < L.$

    If $\nu \geq L,$ then $\alpha \equiv 0 \mod 2^L.$ By \cite[Main Theorem]{JamesMathas1999}, this says that actually $S_2^\lambda$ is simple and equals $D_2^\lambda.$ By \Cref{Young module filtered by Specht Zain}, $Y_2^\lambda$ is filtered by $S_2^{(\lambda \pm d)},$ where $\alpha - 1 + d \supseteq_2 d.$ In this case, $\alpha - 1 \equiv 2^0 + 2^1 + \ldots + 2^{L-1} \mod 2^L,$ so $d$ must only be $0,$ and we have that $Y_2^\lambda = S_2^\lambda = D_2^\lambda.$ In this case, the Young module is not only uniserial but is simple.
    
    If $\nu < L,$ we can write $\alpha \equiv 2^{\nu} + 2^{\nu+1} + \ldots + 2^{L-1} \mod 2^L$. By \Cref{Young module filtered by Specht Zain}, $Y_2^\lambda$ is filtered by $S_2^{(\lambda \pm d)},$ where $\alpha - 1 + d \supseteq_2 d.$ In this case $\alpha - 1 \equiv 2^0 + 2^1 + \ldots + 2^{\nu-1} + 2^{\nu+1} + 2^{\nu+2} + \ldots + 2^{L-1} \mod 2^L$, so $d$ must be $0$ or $2^\nu,$ hence $Y_2^\lambda$ is filtered from bottom to top by $S_2^\lambda$ and $S_2^{(\lambda \pm 2^\nu)}.$ So it suffices to show that $S_2^\lambda$ and $S_2^{(\lambda \pm 2^\nu)}$ have the same composition factors, excluding $D_2^\lambda.$

    By \cite[Corollary 3.10]{kapadia2024submodule}, the composition factors of $S_2^\lambda$ are $D_2^{(\lambda \pm e)},$ where $\alpha + 2e \supseteq_2 e.$ Clearly $e = 0$ satisfies this relation. For $e>0,$ let $e = \sum_{i \geq 0}2^{e_i}$ where the $e_i$ are increasing with $i.$ We have $e_0 \geq \nu,$ otherwise $\alpha + 2e \not\supseteq e.$ Hence $\alpha + 2e \equiv 2^\nu + 2^{\nu+1} + \ldots + 2^{e_0} + \sum_{i \geq 1} 2^{e_i + 1} \mod 2^L.$ If $e$ has at least two non-zero digits, then $\alpha + 2e \not\supseteq e$ as $2^{e_1}$ appears in the right hand side but not the left hand side, and so $e>0$ implies that $e$ must be a power of two. So $\alpha + 2e \supseteq e$ if and only if $e \in \{0\} \cup \{2^j \, | \, \nu \leq j \leq L-1\}.$ Similarly, the composition factors of $S_2^{(\lambda \pm 2^\nu)}$ are $D_2^{(\lambda \pm (2^\nu + f))},$ where $\alpha + 2^{\nu+1} + 2f \supseteq_2 f.$ More precisely, we must have $f \in \{2^j - 2^\nu \, | \, \nu \leq j \leq L-1\}.$ Hence by \Cref{Young module uniserial if composition factors}, $Y_2^\lambda$ is uniserial.
    
    \item Assume $\alpha + 2^\nu \equiv 2^{L-1} \mod 2^L,$ and $2^{\nu+1} + 2^{L-1} > \lambda_2.$  Then $\alpha \equiv 2^\nu + 2^{\nu + 1} + \ldots 2^{L-2}$ mod $2^L.$ By \Cref{Young module filtered by Specht Zain}, $Y_2^\lambda$ has $S_2^{(\lambda \pm 2^{L-1})}$ in its filtration as $\alpha -1 + 2^{L-1} \supseteq_2 2^{L-1}.$ However, $\alpha + 2^L \not\supseteq_2 2^{L-1},$ so $D_2^{(\lambda \pm 2^{L-1})}$ does not appear as a composition factor of $S_2^\lambda.$ Hence by \Cref{Young module not a composition factors of the Specht}, $Y_2^\lambda$ is not uniserial. 

    \item Assume $\alpha - 2^\nu \equiv 0 \mod 2^L.$ If $L = \nu+1,$ then we actually have that $\alpha + 2^\nu \equiv 0 \mod 2^L,$ and is uniserial by the previous calculation. If $L > \nu+1,$ then $\alpha - 1 + 2^{\nu+1} \supseteq_2 2^{\nu+1},$ so by \Cref{Young module filtered by Specht Zain} $Y_2^\lambda$ has $S_2^{(\lambda \pm 2^{\nu+1})}$ in its filtration. However, $\alpha + 2^{\nu+2} \not\supseteq_2 2^{\nu+1},$ so $D_2^{(\lambda \pm 2^{\nu+1})}$ does not appear as a composition factor of $S_2^\lambda.$ Hence by \Cref{Young module not a composition factors of the Specht}, $Y_2^\lambda$ is not uniserial.
\end{enumerate}

Now for the case that $\lambda_1 = \lambda_2.$ We already have by \Cref{Young uniserial 2-singular} that $Y_2^\lambda$ is uniserial if and only if $\lambda_1 =1.$ As $\alpha = 1$ and $2^\nu = 1,$ we see that $\alpha + 2^\nu \equiv 2 \equiv 0 \mod 2^L$ if and only if $\lambda_1 = 1.$

\end{proof}

\newpage
\section{Hook Specht modules which are direct sums of uniserial summands}\label{Hook Specht modules which are direct sums of uniserial summands}

For $\lambda$ a hook partition of $n,$ it is known that if $n$ is even, then $S_2^\lambda$ is indecomposable \cite[Theorem 4.1]{murphy1980decomposability}. Additionally, if $n$ is odd, then hook Specht modules in characteristic $2$ decompose as a direct sum of $2$-part Young modules \cite{DonkinGeranios20201}. We give a preliminary definition and state the theorem precisely below.

\begin{definition}\cite[Definition 3.1]{DonkinGeranios20201}\label{def p-special}
    Let $r,b$ be integers with $r \geq 0$ and $p \geq 0.$ We shall say that the pair $(r,b)$ is $p$-special if
    \begin{enumerate}
        \item $p=0, -r \leq b \leq r$ and $r-b$ is even, or
        \item $p$ is a prime, $r$ has base $p$ expansion $r = \sum_{i=0}^\infty p^ir_i$ and there exists an expression $b = \sum_{i=0}^\infty p^it_i$ with $-r_i \leq t_i \leq r_i$ and $r_i - t_i$ even for all $i\geq 0.$
    \end{enumerate}
\end{definition}

The following theorem is stated in terms of Specht modules for the Iwahori-Hecke algebras.

\begin{theorem}\cite[Theorem 7.1.1]{DonkinGeranios20201}\label{Specht Decompose Young DG}
    Let $a,b \geq 1$ and assume that $a$ and $b$ have different parity. Then we have the following decompositions.
    \begin{enumerate}
        \item For $a$ even and $b$ odd with $a = 2+2u$ and $b = 2v+1$ we have \[S_2^{(a,1^b)} = \bigoplus Y_2^{(2+2c,1+2d)}\] where the sum is over all partitions $\mu = (c,d),$ such that $c+d = u+v$ and $(c-d,u-v)$ is $p$-special.
        \item For $a$ odd and $b$ even with $a=2u+1$ and $b=2v$ we have \[S_2^{(a,1^b)} = \bigoplus Y_2^{(1+2c,2d)}\] where the sum is over all partitions $\mu = (c,d),$ such that $c+d = u+v$ and $(c-d,u-v)$ is $p$-special.
    \end{enumerate}
\end{theorem}

We can specialise the result to the $p=2$ case, which takes us from Specht modules for the Iwahori-Hecke algebra to Specht modules for the symmetric groups, and combine the two parts into one result. Note that due to the self-duality of odd hook Specht modules \cite[Theorem 6.1]{kapadia2024submodule}, there is no loss of generality in assuming that the hook has longer arm than leg.

\begin{corollary}\label{Specht Decompose Young Zain}
    Let $a \geq b \geq 1$ be of different parity, $\lambda = (a,b)$ be a $2$-part partition of $n$ (that is, $a+b = n$), $\alpha := a-b+1,$ and define $D := \{ 0 \leq \delta \leq b \, | \, \alpha-2 + 2\delta \supseteq_2 \delta\}.$ Then \[S_2^{(a,1^b)} = \bigoplus_{\delta \in D} Y_2^{(\lambda \pm \delta)}.\]
\end{corollary}
\begin{proof}
    Assume $a$ is even and $b$ is odd, with $a = 2 + 2u$ and $b = 2v+1$. By \Cref{Specht Decompose Young DG}, we can decompose \[S_2^{(a,1^b)} = \bigoplus Y_2^{(2+2c,1+2d)}\] where the sum is over all partitions $(c,d)$ such that $c+d = u+v$ and $(c-d,u-v)$ is $2$-special. As $(2+2c,1+2d) \trianglerighteq (a,b),$ we can write $(2+2c,1+2d)$ as $(a+\delta, b-\delta),$ where $\delta = 2+2c-a = b-2d-1,$ and we have $c = \frac{a+\delta-2}{2}, d = \frac{b-\delta-1}{2},$ and so $c-d = \frac{a-b-1+2\delta}{2} = \frac{\alpha-2+2\delta}{2}.$ Similarly, we can write $u = \frac{a-2}{2}, v = \frac{b-1}{2}$ and $u-v = \frac{a-b-1}{2} = \frac{\alpha-2}{2}.$ So the direct sum is over all $\delta$ such that $0 \leq \delta \leq b$ and $(c-d,u-v) = (\frac{\alpha-2+2\delta}{2},\frac{\alpha-2}{2})$ is $2$-special.

    In the case that $a$ is odd and $b$ is even, with $a = 2u+1$ and $b = 2v$. Similarly, by \Cref{Specht Decompose Young DG}, we can decompose \[S_2^{(a,1^b)} = \bigoplus Y_2^{(1+2c,2d)}\] where the sum is over all partitions $(c,d)$ such that $c+d = u+v$ and $(c-d,u-v)$ is $2$-special. As $(1+2c,2d) \trianglerighteq (a,b),$ we can write $(1+2c,2d)$ as $(a+\delta, b-\delta),$ where $\delta = 1+2c-a = b-2d,$ and we have $c = \frac{a+\delta-1}{2}, d = \frac{b-\delta}{2},$ and so $c-d = \frac{a-b-1+2\delta}{2} = \frac{\alpha-2+2\delta}{2}.$ Similarly, we can write $u = \frac{a-1}{2}, v = \frac{b}{2}$ and $u-v = \frac{a-b-1}{2} = \frac{\alpha-2}{2}.$ So the direct sum is over all $\delta$ such that $0 \leq \delta \leq b$ and $(c-d,u-v) = (\frac{\alpha-2+2\delta}{2},\frac{\alpha-2}{2})$ is $2$-special.

    By \Cref{def p-special}, this is true if and only if $(\frac{\alpha-2+2\delta}{2})$ has a base $2$ expansion $\sum_{i=0}^\infty 2^ir_i,$ and there exists an expression of $\frac{\alpha-2}{2} = \sum_{i=0}^\infty2^it_i$ where $t_i = 0$ if $r_i = 0,$ and $t_i \in \{1,-1\}$ if $r_i = 1.$ Subtracting the two equations, we get $\delta = \sum_{i=0}^\infty 2^i(r_i - t_i) = \sum_{i=0}^\infty d_i,$ where $d_i = 0$ if $t_i = r_i$ and $d_i = 2$ if $t_i = -1.$ Hence, we have a binary expression for $\delta/2,$ and if there is a $1$ in the $i$th column of $\delta,$ we have $t_i = -1,$ hence $r_i = 1.$ Therefore $\frac{\alpha - 2 + 2\delta}{2} \supseteq_2 \delta/2,$ which is true if and only if $\alpha - 2 + 2\delta \supseteq_2 \delta.$ 
    
\end{proof}

\begin{remark}
With this corollary, we also recover the statement by Murphy \cite[Theorem 4.5]{murphy1980decomposability} that if $n$ is odd, then $S_2^{(\lambda_1, 1^{\lambda_2})}$ is indecomposable if and only if $\alpha - 2 \equiv 0 \mod 2^L.$
\end{remark}
We now have all of the results needed to prove our main theorem.

\begin{theorem}\label{classification_of_DSoUS_hooks}
    Let $\lambda = (\lambda_1, \lambda_2)$ be a partition of an odd integer $n,$ let $\alpha := \lambda_1 - \lambda_2 + 1, \nu := \nu_2(\alpha)$ and $L := L_2(\lambda_2).$ Then the hook Specht module in characteristic $2, S_2^{(\lambda_1, 1^{\lambda_2})}$ is a direct sum of uniserial summands if and only if $\alpha + 2^\nu \equiv 0 \mod 2^L$ and either
    \begin{itemize}
        \item $\nu = 1$ and $\lambda_2 \leq 11$ or;
        \item $\nu = 2$ and $\lambda_2 \leq 21$ or;
        \item $\nu = 3$ and $\lambda_2 \leq 25$ or;
        \item $\nu \geq 4$ and $\lambda_2 \leq 9.$
    \end{itemize}
\end{theorem}
\begin{proof}
    By \Cref{Specht Decompose Young Zain}, we know that \[S_2^{(\lambda_1, 1^{\lambda_2})} \isom \bigoplus Y_2^{(\lambda \pm d)},\] where the sum is over $0 \leq d \leq \lambda_2$ such that $\alpha - 2 + 2d \supseteq_2 d.$ For a fixed $d,$ define $\alpha' := \alpha +2d,\nu' := \nu_2(\alpha + 2d),$ and $L' := L_2(\lambda_2 - d).$

    \begin{itemize}
        \item First assume that $\alpha + 2^\nu \equiv 0 \mod 2^L, \nu = 1$ and $\lambda_2 \leq 11.$ We need to show that for all $d$ such that $0 \leq d \leq \lambda_2$ and $\alpha -2 + 2d \supseteq_2 d,$ we have that $Y_2^{(\lambda \pm d)}$ is uniserial. By \Cref{Uniserial 2-part Young Modules}, we have that $Y_2^\lambda$ is uniserial. We also know that $\alpha \equiv 2^1 + 2^2 + \ldots + 2^{L-1}$ mod $2^L,$ so $\alpha - 2 \equiv 2^2 + 2^3 + \ldots + 2^{L-1}$ mod $2^L,$ so $\alpha -2 + 2d \supseteq d$ if and only if $d \in \{0, 4,8\}$ and $d \leq \lambda_2.$ 
        \begin{itemize}
            \item If $d = 4\leq\lambda_2,$ then $\alpha' \equiv 6$ mod $2^L, \nu' = 1$ so $\alpha' + 2^{\nu'} \equiv 8 \equiv 0 \mod 2^{L'},$ as $\lambda_2 - 4 \leq 7,$ so $2^{L'} \leq 8.$ Hence $Y_2^{(\lambda \pm 4)}$ is uniserial.
            
            \item If $d = 8\leq\lambda_2,$ then $\alpha' \equiv 14$ mod $2^L, \nu' = 1$ so $\alpha' + 2^{\nu'} \equiv 16 \equiv 0 \mod 2^{L'},$ as $\lambda_2 - 8 \leq 3,$ so $2^{L'} \leq 4.$ Hence $Y_2^{(\lambda \pm 8)}$  is uniserial
        \end{itemize}
        Therefore $S_2^{(\lambda_1, 1^{\lambda_2})}$ is a direct sum of uniserial summands.

        \item Now assume that $\alpha + 2^\nu \equiv 0 \mod 2^L, \nu = 2$ and $\lambda_2 \leq 21.$ As before, we need to show that for all $d$ such that $0 \leq d \leq \lambda_2$ and $\alpha -2 + 2d \supseteq_2 d,$ we have that $Y_2^{(\lambda \pm d)}$ is uniserial, and by \Cref{Uniserial 2-part Young Modules}, we have that $Y_2^\lambda$ is uniserial. We also know that $\alpha \equiv 2^2 + 2^3 + \ldots + 2^{L-1}$ mod $2^L,$ so $\alpha - 2 \equiv 2^1 + 2^3 + 2^4 + \ldots + 2^{L-1}$ mod $2^L,$ so $\alpha - 2 + 2d \supseteq_2 d$ if and only if $d \in \{0, 2, 6, 8, 10, 16, 18\}$ and $d \leq \lambda_2.$
        \begin{itemize}
            \item If $d=2\leq \lambda_2,$ then $\alpha' \equiv 0$ mod $2^L,$ so $\alpha' + 2^{\nu'} \equiv 0$ mod $2^{L'}$ as required, and $Y_2^{(\lambda \pm 2)}$ is uniserial.
            
            \item If $d= 6\leq\lambda_2,$ then $\alpha' \equiv 8 \mod 2^L, \nu' = 3$ so $\alpha' + 2^{\nu'} \equiv 16 \equiv 0 $ mod $2^{L'},$ as $\lambda_2 - 6 \leq 15,$ so $2^{L'} \leq 16.$ Hence $Y_2^{(\lambda \pm 6)}$ is uniserial.

            \item If $d=8\leq \lambda_2,$ then $\alpha' \equiv 12 \mod 2^L, \nu' = 2$ so $\alpha' + 2^{\nu'} \equiv 16 \equiv 0$ mod $2^{L'},$ as $\lambda_2 - 8 \leq 13,$ so $2^{L'} \leq 16.$ Hence $Y_2^{(\lambda \pm 8)}$ is uniserial.

            \item If $d=10\leq \lambda_2,$ then $\alpha' \equiv 16 \mod 2^L, \nu' = 4$ so $\alpha' + 2^{\nu'} \equiv 32 \equiv 0$ mod $2^{L'},$ as $\lambda_2 - 10 \leq 11,$ so $2^{L'} \leq 16.$ Hence $Y_2^{(\lambda \pm 10)}$ is uniserial.

            \item If $d=16\leq \lambda_2,$ then $\alpha' \equiv 28 \mod 2^L, \nu' = 2$ so $\alpha' + 2^{\nu'} \equiv 32 \equiv 0$ mod $2^{L'},$ as $\lambda_2 - 16 \leq 5,$ so $2^{L'} \leq 8.$ Hence $Y_2^{(\lambda \pm 16)}$ is uniserial.

            \item If $d=18\leq \lambda_2,$ then $\alpha' \equiv 32 \mod 2^L,$ but $\lambda_2 \leq 21$ so $2^L \leq 32,$ so $\alpha' \equiv 0 \mod 2^L$ and so $\alpha' + 2^{\nu'} \equiv 0 \mod 2^{L'}.$ Hence $Y_2^{(\lambda \pm 18)}$ is uniserial.
        \end{itemize}
        Therefore $S_2^{(\lambda_1, 1^{\lambda_2})}$ is a direct sum of uniserial summands.

        \item Now assume that $\alpha + 2^\nu \equiv 0 \mod 2^L, \nu = 3$ and $\lambda_2 \leq 25.$ As before, we need to show that for all $d$ such that $0 \leq d \leq \lambda_2$ and $\alpha -2 + 2d \supseteq_2 d,$ we have that $Y_2^{(\lambda \pm d)}$ is uniserial, and by \Cref{Uniserial 2-part Young Modules}, we have that $Y_2^\lambda$ is uniserial. We also know that $\alpha \equiv 2^3 + 2^4 + \ldots + 2^{L-1}$ mod $2^L,$ so $\alpha - 2 \equiv 2^1 + 2^2 + 2^4 + 2^5 + \ldots + 2^{L-1}$ mod $2^L,$ so $\alpha - 2 + 2d \supseteq_2 d$ if and only if $d \in \{0, 2, 4, 10, 12, 18, 20\}$ and $d \leq \lambda_2.$

        \begin{itemize}
            \item If $d=2 \leq \lambda_2,$ then $\alpha' \equiv 2^2 + 2^3 + \ldots + 2^{L-1} \mod 2^L,\nu' =2$ so $\alpha' + 2^{\nu'} \equiv 0 \mod 2^L,$ so is equivalent to $0$ mod $2^{L'}.$ Hence $Y_2^{(\lambda \pm 2)}$ is uniserial.

            \item If $d=4\leq \lambda_2,$ then $\alpha' \equiv 0 \mod 2^L,$ so $\alpha' + 2^{\nu'} \equiv 0 \mod 2^{L'}.$ Hence $Y_2^{(\lambda \pm 4)}$ is uniserial.

            \item If $d = 10 \leq \lambda_2,$ then $\alpha' \equiv 12 \mod 2^L, \nu' = 2$ so $\alpha' + 2^{\nu'} \equiv 16 \equiv \mod 2^{L'},$ as $\lambda_2 -10 \leq 15,$ so $2^{L'} \leq 16.$ Hence $Y_2^{(\lambda \pm 10)}$ is uniserial.

            \item If $d = 12 \leq \lambda_2,$ then $\alpha' \equiv 16 \mod 2^L,$ so $\alpha' + 2^{\nu'} \equiv 0 \mod 2^{L'}$ as $\lambda_2 - 12 \leq 13,$ so $2^{L'} \leq 16.$ Hence $Y_2^{(\lambda \pm 12)}$ is uniserial.

            \item If $d = 18\leq \lambda_2,$ then $\alpha' \equiv 28 \mod 2^L, \nu' = 2$ so $\alpha' + 2^{\nu'} \equiv 32 \equiv 0 \mod 2^{L'},$ as $\lambda_2 - 18 \leq 7,$ so $2^{L'} \leq 8.$ Hence $Y_2^{(\lambda \pm 18)}$ is uniserial.

            \item If $d = 20 \leq \lambda_2,$ then $\alpha' \equiv 32 \equiv \mod 2^L$ as $\lambda_2 \leq 25$ so $2^{L} \leq 32.$ So $\alpha' + 2^{\nu'} \equiv 0 \mod 2^{L'}.$ Hence $Y_2^{(\lambda \pm 20)}$ is uniserial.
        \end{itemize}
        Therefore $S_2^{(\lambda_1, 1^{\lambda_2})}$ is a direct sum of uniserial summands.
    \end{itemize}

    \item Now assume that $\alpha + 2^\nu \equiv 0 \mod 2^L, \nu \geq 4$ and $\lambda_2 \leq 9.$ As before, we need to show that for all $d$ such that $0 \leq d \leq \lambda_2$ and $\alpha -2 + 2d \supseteq_2 d,$ we have that $Y_2^{(\lambda \pm d)}$ is uniserial, and by \Cref{Uniserial 2-part Young Modules}, we have that $Y_2^\lambda$ is uniserial. We also know that $\alpha \equiv 0$ mod $2^L,$ so $\alpha - 2 \equiv 2^1 + 2^2 + 2^3$ mod $2^L,$ so $\alpha - 2 + 2d \supseteq_2 d$ if and only if $d \in \{0, 2, 4, 8\}$ and $d \leq \lambda_2.$

    \begin{itemize}
        \item If $d =2 \leq \lambda_2,$ then $\alpha' \equiv 4 \mod 2^L, \nu' = 2$ so $\alpha' + 2^{\nu'} \equiv 8 \equiv 0 \mod 2^{L'},$ as $\lambda_2 - 2 \leq 7,$ so $2^{L'} \leq 8.$ Hence $Y_2^{(\lambda \pm 2)}$ is uniserial.

        \item If $d = 4 \leq \lambda_2,$ then $\alpha' \equiv 8 \mod 2^L, \nu' = 3$ so $\alpha' + 2^{\nu'} \equiv 16 \equiv 0 \mod 2^{L'},$ as $\lambda_2 - 4 \leq 5,$ so $2^{L'} \leq 8.$ Hence $Y_2^{(\lambda \pm 4)}$ is uniserial.

        \item If $d= 8 \leq \lambda_2,$ then $\alpha' \equiv 16 \equiv 0 \mod 2^L$ as $\lambda_2 \leq 9$ so $2^L \leq 16.$ So $\alpha' + 2^{\nu'} \equiv 0 \mod 2^{L'}.$ Hence $Y_2^{(\lambda_2 \pm 8)}$ is uniserial.
    \end{itemize}
    Therefore $S_2^{(\lambda_1, 1^{\lambda_2})}$ is a direct sum of uniserial summands.

    We now prove the converse. Clearly, if $\alpha + 2^{\nu} \not\equiv 0 \mod 2^L,$ then $Y_2^\lambda$ is not uniserial, so $S_2^{(\lambda_1, 1^{\lambda_2})}$ is not a direct sum of uniserial summands. So we can assume that $\alpha + 2^\nu \equiv 0 \mod 2^L.$ We show that for each value of $\nu,$ if $\lambda_2$ does not satisfy the appropriate inequality, then we get a non-uniserial Young module in the decomposition of $S_2^{(\lambda_1, 1^{\lambda_2})}.$

    \begin{itemize}
        \item First assume that $\nu = 1$ and $\lambda_2 \geq 12.$ Then $\alpha \equiv 2^1 + 2^2 + \ldots + 2^{L-1} \mod 2^L,$ so $\alpha - 2 \equiv 2^2 + 2^3 + \ldots + 2^{L-1} \mod 2^L,$ so $\alpha -2 + 2(4) \supseteq_2 4$ and $Y_2^{(\lambda \pm 4)}$ is a summand of $S_2^{(\lambda_1, 1^{\lambda_2})}.$ However, $\alpha' \equiv 6 \mod 2^L, \nu = 1,$ so $\alpha' + 2^{\nu'} \equiv 8 \not\equiv 0 \mod 2^{L'},$ as $\lambda_2 - 4 \geq 8,$ so $2^{L'} \geq 16.$ Hence $Y_2^{(\lambda \pm 4)}$ is not uniserial, and $S_2^{(\lambda_1, 1^{\lambda_2})}$ is not a direct sum of uniserial summands.

        \item Now assume that $\nu = 2$ and $\lambda_2 \geq 22.$ Then $\alpha \equiv 2^2 + 2^3 + \ldots + 2^{L-1} \mod 2^L,$ so $\alpha - 2 \equiv 2^1 + 2^3 + 2^4 + \ldots + 2^{L-1} \mod 2^L,$ so $\alpha - 2 + 2(6) \supseteq_2 6$ and $Y_2^{(\lambda \pm 6)}$ is a summand of $S_2^{(\lambda_1, 1^{\lambda_2})}.$ However, $\alpha' \equiv 8 \mod 2^L, \nu = 3,$ so $\alpha' + 2^{\nu'} \equiv 16 \not\equiv 0 \mod 2^{L'},$ as $\lambda_2 - 6 \geq 16,$ so $2^{L'} \geq 32.$ Hence $Y_2^{(\lambda \pm 6)}$ is not uniserial, and $S_2^{(\lambda_1, 1^{\lambda_2})}$ is not a direct sum of uniserial summands.

        \item Now assume that $\nu = 3$ and $\lambda_2 \geq 26.$ Then $\alpha \equiv 2^3 + 2^4 + \ldots + 2^{L-1} \mod 2^L,$ so $\alpha - 2 \equiv 2^1 + 2^2 + 2^4 + 2^5 + \ldots + 2^{L-1} \mod 2^L,$ so $\alpha -2 + 2(10) \supseteq_2 10$ and $Y_2^{(\lambda \pm 10)}$ is a summand of $S_2^{(\lambda_1, 1^{\lambda_2})}.$ However, $\alpha' \equiv 12 \mod 2^L, \nu = 2$ so $\alpha' + 2^{\nu'} \equiv 16 \not\equiv 0 \mod 2^{L'}$ as $\lambda_2 - 10 \geq 16,$ so $2^{L'} \geq 32.$ Hence $Y_2^{(\lambda \pm 10)}$ is not uniserial, and $S_2^{(\lambda_1, 1^{\lambda_2})}$ is not a direct sum of uniserial summands.

        \item Now assume that $\nu \geq 4$ and $\lambda_2 \geq 10.$ Then $\alpha \equiv 2^4 + 2^5 + \ldots + 2^{L-1} \mod 2^L,$ so $\alpha - 2 \equiv 2^1 + 2^2 + 2^3 + 2^5 + 2^6 + \ldots + 2^{L-1} \mod 2^L,$ so $\alpha - 2 + 2(2) \supseteq_2 2$ and $Y_2^{(\lambda \pm 2)}$ is a direct summand of $S_2^{(\lambda_1, 1^{\lambda_2})}.$ However, $\alpha' \equiv 2^2 + 2^4 + 2^5 + \ldots + 2^{L-1} \mod 2^L, \nu = 2$  so $\alpha' + 2^{\nu'} \equiv 2^3 + 2^4 + \ldots + 2^{L-1} \not\equiv 0 \mod 2^{L'}$ as $\lambda_2 - 2 \geq 8$ so $2^{L'} \geq 16.$ Hence $Y_2^{(\lambda \pm 2)}$ is not uniserial, and $S_2^{(\lambda_1, 1^{\lambda_2})}$ is not a direct sum of uniserial summands.
    \end{itemize}

    Hence the required result.
\end{proof}

\begin{example}
    By \Cref{classification_of_DSoUS_hooks}, we see that if the hook Specht module is a direct sum of uniserial summands, then the maximum number of summands is seven; when $\nu =2$ and $18 \leq \lambda_2 \leq 21,$ or when $\nu = 3$ and $20 \leq \lambda_2 \leq 25.$ We take the maximal value for $\lambda_2$ in either case, and give the full submodule structure of each of the uniserial direct summands below. 

    In the diagrams below, the vertices represent submodules of $Y_2^\lambda$. The directed edge from a vertex $u$ to a vertex $v$ indicates that the submodule associated to $u$ is a maximal submodule of the submodule associated to $v,$ and the edges are labelled by $d$ such that $D_2^{(\lambda \pm d)}$ is the irreducible subquotient.

    The Specht module $S_2^{(48,1^{21})}$ has $\lambda = (48,21), \alpha = 48 - 21 + 1 = 28, \nu = 2, L = 32, \alpha + 2^\nu \equiv 0 \mod 2^L,$ and $\lambda_2 \leq 21.$ Hence $S_2^{(48, 1^{21})}$ is a direct sum of the following seven uniserial Young modules.

    \resizebox{1\textwidth}{!}{
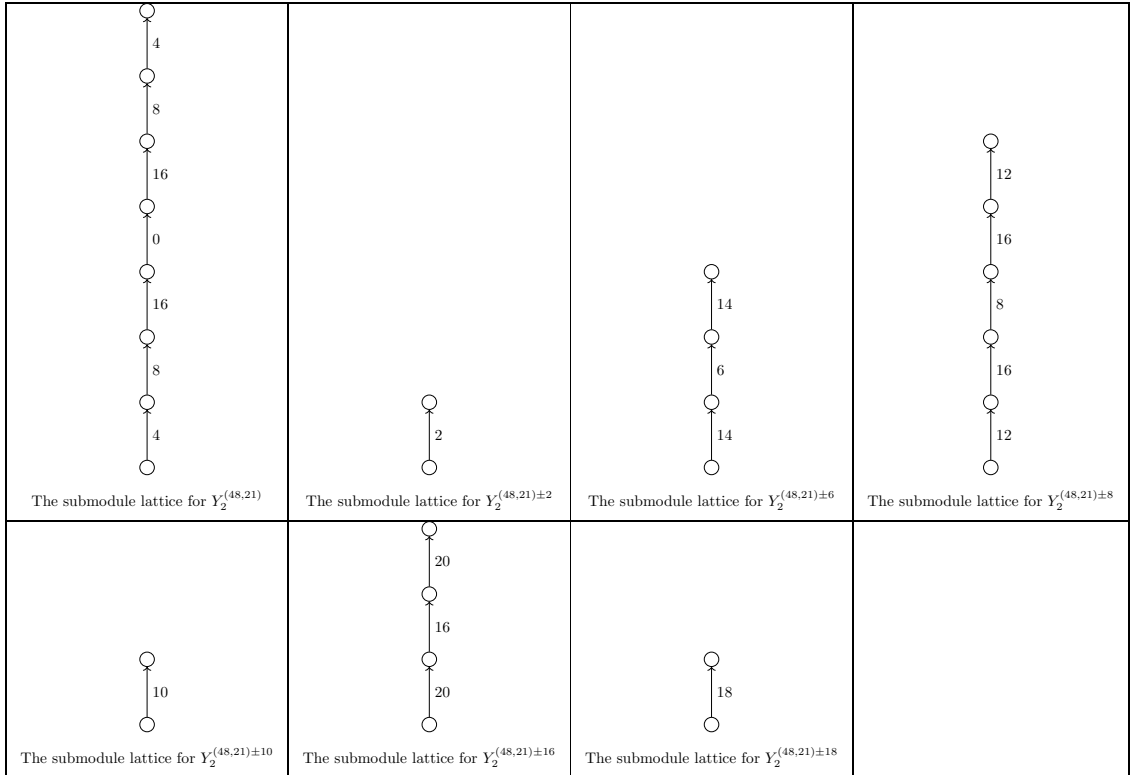
\begin{tabular}{|c|c|c|c|}
     \hline
\begin{tikzpicture}[main/.style = {draw, circle}, scale = 1.5]
\node at (0,-1/2) {The submodule lattice for $Y_2^{( 48,21 )}$};

\node[main] (0) at (0,0) {};
\node[main] (1) at (0,1) {};
\node[main] (2) at (0,2) {};
\node[main] (3) at (0,3) {};
\node[main] (4) at (0,4) {};
\node[main] (5) at (0,5) {};
\node[main] (6) at (0,6) {};
\node[main] (7) at (0,7) {};

\draw[->] (0) edge ["$4$"', pos = 0.5] (1);
\draw[->] (1) edge ["$8$"', pos = 0.5] (2);
\draw[->] (2) edge ["$16$"', pos = 0.5] (3);
\draw[->] (3) edge ["$0$"', pos = 0.5] (4);
\draw[->] (4) edge ["$16$"', pos = 0.5] (5);
\draw[->] (5) edge ["$8$"', pos = 0.5] (6);
\draw[->] (6) edge ["$4$"', pos = 0.5] (7);

\end{tikzpicture}
     &
\begin{tikzpicture}[main/.style = {draw, circle}, scale = 1.5]
\node at (0,-1/2) {The submodule lattice for $Y_2^{( 48,21 )\pm 2}$};

\node[main] (0) at (0,0) {};
\node[main] (1) at (0,1) {};

\draw[->] (0) edge ["$2$"', pos = 0.5] (1);

\end{tikzpicture}
     &
\begin{tikzpicture}[main/.style = {draw, circle}, scale = 1.5]
\node at (0,-1/2) {The submodule lattice for $Y_2^{( 48,21 )\pm 6}$};

\node[main] (0) at (0,0) {};
\node[main] (1) at (0,1) {};
\node[main] (2) at (0,2) {};
\node[main] (3) at (0,3) {};

\draw[->] (0) edge ["$14$"', pos = 0.5] (1);
\draw[->] (1) edge ["$6$"', pos = 0.5] (2);
\draw[->] (2) edge ["$14$"', pos = 0.5] (3);

\end{tikzpicture}
     &
\begin{tikzpicture}[main/.style = {draw, circle}, scale = 1.5]
\node at (0,-1/2) {The submodule lattice for $Y_2^{( 48,21 )\pm 8}$};

\node[main] (0) at (0,0) {};
\node[main] (1) at (0,1) {};
\node[main] (2) at (0,2) {};
\node[main] (3) at (0,3) {};
\node[main] (4) at (0,4) {};
\node[main] (5) at (0,5) {};

\draw[->] (0) edge ["$12$"', pos = 0.5] (1);
\draw[->] (1) edge ["$16$"', pos = 0.5] (2);
\draw[->] (2) edge ["$8$"', pos = 0.5] (3);
\draw[->] (3) edge ["$16$"', pos = 0.5] (4);
\draw[->] (4) edge ["$12$"', pos = 0.5] (5);

\end{tikzpicture}
     \\ 
     \hline
\begin{tikzpicture}[main/.style = {draw, circle}, scale = 1.5]
\node at (0,-1/2) {The submodule lattice for $Y_2^{( 48,21 )\pm 10}$};

\node[main] (0) at (0,0) {};
\node[main] (1) at (0,1) {};

\draw[->] (0) edge ["$10$"', pos = 0.5] (1);

\end{tikzpicture}
     &
\begin{tikzpicture}[main/.style = {draw, circle}, scale = 1.5]
\node at (0,-1/2) {The submodule lattice for $Y_2^{( 48,21 )\pm 16}$};

\node[main] (0) at (0,0) {};
\node[main] (1) at (0,1) {};
\node[main] (2) at (0,2) {};
\node[main] (3) at (0,3) {};

\draw[->] (0) edge ["$20$"', pos = 0.5] (1);
\draw[->] (1) edge ["$16$"', pos = 0.5] (2);
\draw[->] (2) edge ["$20$"', pos = 0.5] (3);

\end{tikzpicture}
     &
\begin{tikzpicture}[main/.style = {draw, circle}, scale = 1.5]
\node at (0,-1/2) {The submodule lattice for $Y_2^{( 48,21 )\pm 18}$};

\node[main] (0) at (0,0) {};
\node[main] (1) at (0,1) {};

\draw[->] (0) edge ["$18$"', pos = 0.5] (1);

\end{tikzpicture}
     &
\begin{tikzpicture}

\end{tikzpicture}
     \\
     \hline
\end{tabular}
}
\\
The Specht module $S_2^{(48,1^{25})}$ has $\lambda = (48,25), \alpha = 48 - 25 + 1 = 24, \nu = 3, L = 32, \alpha + 2^\nu \equiv 0 \mod 2^L,$ and $\lambda_2 \leq 25.$ Hence $S_2^{(48, 1^{25})}$ is a direct sum of the following seven uniserial Young modules.
\\
\resizebox{1\textwidth}{!}{
\begin{tabular}{|c|c|c|c|}
     \hline
\begin{tikzpicture}[main/.style = {draw, circle}, scale = 1.5]
\node at (0,-1/2) {The submodule lattice for $Y_2^{( 48,25 )}$};

\node[main] (0) at (0,0) {};
\node[main] (1) at (0,1) {};
\node[main] (2) at (0,2) {};
\node[main] (3) at (0,3) {};
\node[main] (4) at (0,4) {};
\node[main] (5) at (0,5) {};

\draw[->] (0) edge ["$8$"', pos = 0.5] (1);
\draw[->] (1) edge ["$16$"', pos = 0.5] (2);
\draw[->] (2) edge ["$0$"', pos = 0.5] (3);
\draw[->] (3) edge ["$16$"', pos = 0.5] (4);
\draw[->] (4) edge ["$8$"', pos = 0.5] (5);

\end{tikzpicture}
     &
\begin{tikzpicture}[main/.style = {draw, circle}, scale = 1.5]
\node at (0,-1/2) {The submodule lattice for $Y_2^{( 48,25 )\pm 2}$};

\node[main] (0) at (0,0) {};
\node[main] (1) at (0,1) {};
\node[main] (2) at (0,2) {};
\node[main] (3) at (0,3) {};
\node[main] (4) at (0,4) {};
\node[main] (5) at (0,5) {};
\node[main] (6) at (0,6) {};
\node[main] (7) at (0,7) {};

\draw[->] (0) edge ["$6$"', pos = 0.5] (1);
\draw[->] (1) edge ["$10$"', pos = 0.5] (2);
\draw[->] (2) edge ["$18$"', pos = 0.5] (3);
\draw[->] (3) edge ["$0$"', pos = 0.5] (4);
\draw[->] (4) edge ["$18$"', pos = 0.5] (5);
\draw[->] (5) edge ["$10$"', pos = 0.5] (6);
\draw[->] (6) edge ["$6$"', pos = 0.5] (7);

\end{tikzpicture}
     &
\begin{tikzpicture}[main/.style = {draw, circle}, scale = 1.5]
\node at (0,-1/2) {The submodule lattice for $Y_2^{( 48,25 )\pm 4}$};

\node[main] (0) at (0,0) {};
\node[main] (1) at (0,1) {};

\draw[->] (0) edge ["$0$"', pos = 0.5] (1);

\end{tikzpicture}
     &
\begin{tikzpicture}[main/.style = {draw, circle}, scale = 1.5]
\node at (0,-1/2) {The submodule lattice for $Y_2^{( 48,25 )\pm 10}$};

\node[main] (0) at (0,0) {};
\node[main] (1) at (0,1) {};
\node[main] (2) at (0,2) {};
\node[main] (3) at (0,3) {};
\node[main] (4) at (0,4) {};
\node[main] (5) at (0,5) {};

\draw[->] (0) edge ["$14$"', pos = 0.5] (1);
\draw[->] (1) edge ["$18$"', pos = 0.5] (2);
\draw[->] (2) edge ["$10$"', pos = 0.5] (3);
\draw[->] (3) edge ["$18$"', pos = 0.5] (4);
\draw[->] (4) edge ["$14$"', pos = 0.5] (5);

\end{tikzpicture}
     \\ 
     \hline
\begin{tikzpicture}[main/.style = {draw, circle}, scale = 1.5]
\node at (0,-1/2) {The submodule lattice for $Y_2^{( 48,25 )\pm 12}$};

\node[main] (0) at (0,0) {};
\node[main] (1) at (0,1) {};

\draw[->] (0) edge ["$12$"', pos = 0.5] (1);

\end{tikzpicture}
     &
\begin{tikzpicture}[main/.style = {draw, circle}, scale = 1.5]
\node at (0,-1/2) {The submodule lattice for $Y_2^{( 48,25 )\pm 18}$};

\node[main] (0) at (0,0) {};
\node[main] (1) at (0,1) {};
\node[main] (2) at (0,2) {};
\node[main] (3) at (0,3) {};

\draw[->] (0) edge ["$22$"', pos = 0.5] (1);
\draw[->] (1) edge ["$18$"', pos = 0.5] (2);
\draw[->] (2) edge ["$22$"', pos = 0.5] (3);

\end{tikzpicture}
     &
\begin{tikzpicture}[main/.style = {draw, circle}, scale = 1.5]
\node at (0,-1/2) {The submodule lattice for $Y_2^{( 48,25 )\pm 20}$};

\node[main] (0) at (0,0) {};
\node[main] (1) at (0,1) {};

\draw[->] (0) edge ["$20$"', pos = 0.5] (1);

\end{tikzpicture}
     &
\begin{tikzpicture}

\end{tikzpicture}
     \\
     \hline
\end{tabular}
}

\end{example}

\newpage
\printbibliography
\end{document}